\documentclass[12pt]{amsart}
\usepackage{amsfonts}
\usepackage{bbm}
\usepackage{mathrsfs}
\usepackage{txfonts}
\usepackage{amssymb}
\usepackage{graphicx, epsf,amssymb,amscd,amsfonts,times}
\usepackage{color,xspace,colortbl}
\usepackage{mathtools}
\usepackage{overpic}

\newtheorem{thm}{Theorem}[section]

\newtheorem{prop}[thm]{Proposition}
\newtheorem{defn}[thm]{Definition}
\newtheorem{lemma}[thm]{Lemma}
\newtheorem{cor}[thm]{Corollary}

\theoremstyle{remark}
\newtheorem{rmk}[thm]{Remark}

\newcommand{\s}{\vskip.1in}
\newcommand{\n}{\noindent}
\newcommand{\bdry}{\partial}

\newcommand{\be}{\begin{enumerate}}
\newcommand{\ee}{\end{enumerate}}

\numberwithin{equation}{subsection}

\begin{document}

\title{Bypass attachments and homotopy classes of 2-plane fields in contact topology}

\author{Yang Huang}
\address{University of Southern California, Los Angeles, CA 90089}
\email{huangyan@usc.edu}

\begin{abstract}
We use the generalized Pontryagin-Thom construction to analyze the effect of attaching a bypass on the homotopy class of the contact structure. In particular, given a 3-dimensional contact manifold with convex boundary, we show that the bypass triangle attachment changes the homotopy class of the contact structure relative to the boundary, and the difference is measured by the Hopf invariant.
\end{abstract}

\maketitle

The goal of this paper is to study how bypass attachments affect the homotopy type of the contact structure on a given contact manifold with convex boundary. Although the notion of a {\em bypass} was defined by K. Honda in \cite{Ho1} and has been used in various classification problems in 3-dimensional contact geometry, it has not been clear until now how this operation changes the homotopy class of the underlying 2-plane field distribution. In particular, we will see in this paper how a special sequence of bypass attachments, namely, a {\em bypass triangle attachment}, affects the homotopy type of the contact structure.

Let $M$ be an compact oriented 3-manifold with boundary. Let $\xi$ and $\xi'$ be two co-oriented contact structures on $M$ such that $\xi=\xi'$ in the complement of an open ball $B^3 \subset int(M)$. Using a generalization of the Pontryagin-Thom construction for compact manifolds with boundary, we define a 3-dimensional obstruction class $o_3(\xi,\xi') \in \mathbb{Z}/d(\xi)$, where $d(\xi)$ is the divisibility of the Euler class $e(\xi)=e(\xi') \in H^2(M,\mathbb{Z})$, and use it to distinguish homotopy classes of $\xi$ and $\xi'$.

In order to state the main result of this paper, we first need to define a bypass. Let $\Sigma$ be a convex surface, and $\alpha$ be a Legendrian arc on $\Sigma$ which intersects the dividing set $\Gamma_\Sigma$ in three points. According to~\cite{Ho1}, a {\em bypass} along $\alpha$ on $\Sigma$ is half of an overtwisted disk whose boundary is the union of two Legendrian arcs $\alpha\cup\beta$, where the {\em Thurston-Bennequin invariants}\footnote{By fixing framing at endpoints, the Thurston-Bennequin invariant is well-defined for Legendrian arcs.} of $\alpha$ and $\beta$ are $-1$ and 0, respectively. See Section 1 for the construction of a bypass attachment along $\alpha$, which we denote by $\sigma_\alpha$. We note here that $\sigma_\alpha$ locally changes the dividing set in a neighborhood of $\alpha$ as depicted in Figure~\ref{BypassAtt}.

\begin{figure}[h]
    \begin{overpic}[scale=.55]{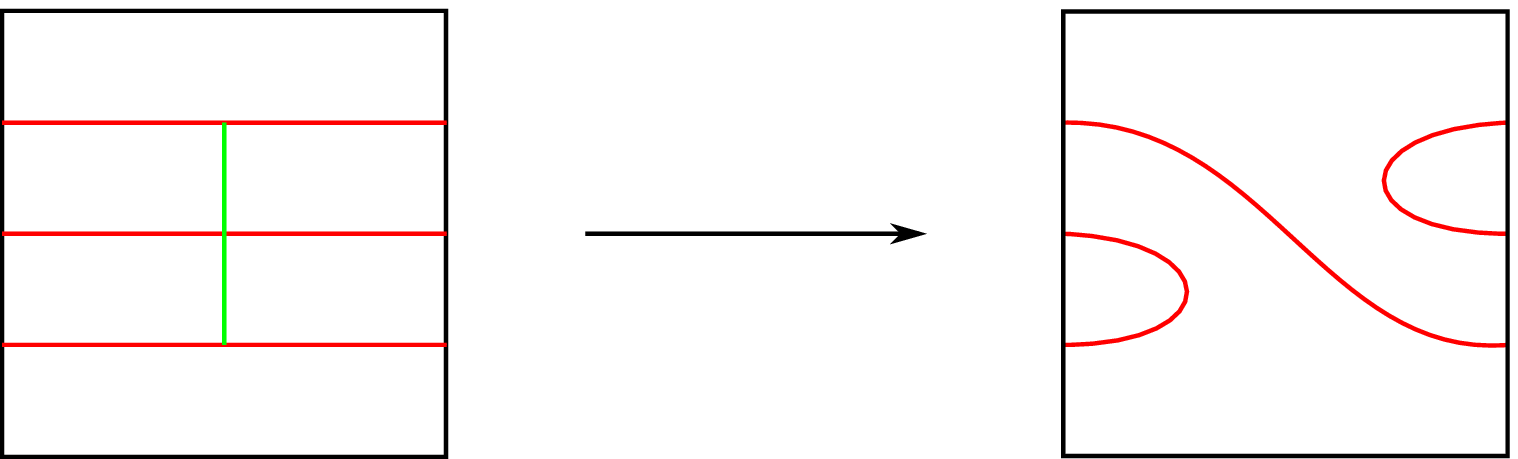}
    \put(16,10){$\alpha$}
    \put(48,17){$\sigma_\alpha$}
    \end{overpic}
    \caption{}
    \label{BypassAtt}
\end{figure}

In this paper, we study the effect of a bypass attachment on the homotopy class of the contact structure. Namely, by making several choices, we compute the {\em relative Pontryagin submanifold} for a bypass attachment as follows. The definition of relative Pontryagin submanifold is discussed in Section 2.

Let $V=[-3/4,3/4]\times[-1,1]\times[0,1] \subset \mathbb{R}^3$ be a 3-manifold with boundary equipped with the standard coordinates, and $\xi$ be a contact structure on $V$ defined by $\xi=ker~\lambda$, where $\lambda = \cos(2\pi x)dy - \sin(2\pi x)dz$. Let $\alpha=[-1/2,1/2]\times\{0\}\times\{1\}$ be a Legendrian arc. We denote by $\xi\ast\sigma_\alpha$ the contact structure given by a bypass attachment to $\xi$ along $\alpha$. See Section 3 for the explicit construction of $(V,\xi\ast\sigma_\alpha)$. Trivialize $TV$ by the standard embedding $V \subset \mathbb{R}^3$ and look at the associated Gauss map $G_{\xi\ast\sigma_\alpha}: V \to S^2$. Observe that $p=(1,0,0) \in S^2$ is a regular value of $G_{\xi\ast\sigma_\alpha}$ by construction.

\begin{thm} \label{MainThm1}
Let $(V,\xi\ast\sigma_\alpha)$ be the contact manifold described above. Then the Pontryagin submanifold $G^{-1}_{\xi\ast\sigma_\alpha}(p) \subset V$ is a properly embedded framed arc with framing as depicted in Figure~\ref{thm1}.

\begin{figure}[h]
    \begin{overpic}[scale=.25]{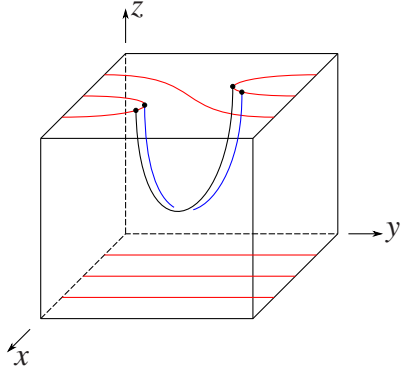}
    \put(1.5,-3.5){$x$}
    \put(101,31){$y$}
    \put(33,90){$z$}
    \end{overpic}
    \caption{The Pontryagin submanifold $G^{-1}_{\xi\ast\sigma_\alpha}(p)$ in $V$. The blue arc is a parallel copy of $G^{-1}_{\xi\ast\sigma_\alpha}(p)$ which defines the framing.}
    \label{thm1}
\end{figure}
\end{thm}

\begin{rmk}
The Pontryagin submanifold $G^{-1}_{\xi\ast\sigma_\alpha}(p)$ in Theorem~\ref{MainThm1} depends on various choices including the trivialization of $TV$ and the regular value $p$. For example, it will be clear from the proof of Theorem~\ref{MainThm1} that $q=(-1,0,0) \in S^2$ is also a regular value of $G_{\xi\ast\sigma_\alpha}$, but $G^{-1}_{\xi\ast\sigma_\alpha}(q)$ is the empty set.
\end{rmk}

The following corollary follows immediately from Theorem~\ref{MainThm1} by the local nature of the bypass attachment.

\begin{cor}
Let $(M,\xi)$ be a contact 3-manifold with convex boundary and $\alpha \subset \bdry M$ be a Legendrian arc along which a bypass can be attached. Then there exists a trivialization of $TM$ and a common regular value $p \in S^2$ of $G_\xi$ and $G_{\xi\ast\sigma_\alpha}$ such that the Pontryagin submanifold $G^{-1}_{\xi\ast\sigma_\alpha}(p) = G^{-1}_\xi(p) \cup \gamma$, where $G^{-1}_\xi(p)$ is the Pontryagin submanifold associated with $\xi$ and $\gamma \subset M$ is a properly embedded framed arc as depicted in Figure~\ref{thm1} which does not link $G^{-1}_\xi(p)$.
\end{cor}

As an application, we study the effect of a bypass triangle attachment on the homotopy class of the contact structure. We first define a bypass triangle attachment as follows.

\begin{defn}
Let $(M,\xi)$ be a contact 3-manifold with convex boundary and $\alpha\subset\bdry M$ be a Legendrian arc. A {\em bypass triangle attachment} along $\alpha$ is the composition of three bypass attachments along Legendrian arcs $\alpha$, $\alpha'$ and $\alpha''$ as depicted in Figure~\ref{BypassTri}. We denote the bypass triangle attachment along $\alpha$ by $\triangle_\alpha=\sigma_\alpha\ast\sigma_{\alpha'}\ast\sigma_{\alpha''}$, where the composition $\ast$ is from left to right, i.e., we attach $\sigma_\alpha$ first, followed by $\sigma_{\alpha'}$ and then $\sigma_{\alpha''}$.
\end{defn}

\begin{figure}[h]
    \begin{overpic}[scale=.38]{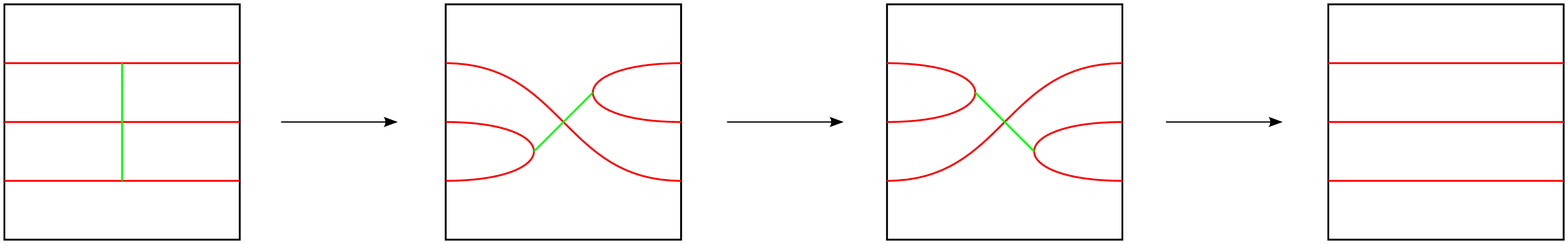}
    \put(8,5.2){\small{$\alpha$}}
    \put(35.3,9.2){\small{$\alpha'$}}
    \put(63.3,9){\small{$\alpha''$}}
    \put(20,9){\small{$\sigma_\alpha$}}
    \put(48.5,9){\small{$\sigma_{\alpha'}$}}
    \put(76.6,9){\small{$\sigma_{\alpha''}$}}
    \put(6,-4){(a)}
    \put(34.2,-4){(b)}
    \put(63,-4){(c)}
    \put(91,-4){(d)}
    \end{overpic}
    \newline
    \caption{}
    \label{BypassTri}
\end{figure}

It follows from Giroux's Flexibility Theorem (c.f. Theorem~\ref{Flex}) that a bypass triangle attachment does not change the contact structure in a neighborhood of $\bdry M$ up to isotopy. In fact, it only affects the contact structure within a ball embedded in the interior of $M$, which can be measured by a 3-dimensional obstruction class $o_3$ defined in Section 2. Now we state the following theorem for the homotopy class of a bypass triangle attachment.

\begin{thm} \label{MainThm2}
If $(M,\xi)$ is a contact manifold with convex boundary, and $\xi'$ is the contact structure obtained from $\xi$ by attaching a bypass triangle on $\bdry M$, then $o_3(\xi,\xi')=-1$. In particular, $\xi'$ is not homotopic to $\xi$ relative to the boundary as 2-plane field distributions.
\end{thm}

\begin{rmk}
Theorem~\ref{MainThm2} is an important ingredient in the analysis of the universal cover of a contact category $\mathscr{C}(\Sigma)$ defined in~\cite{Ho3}, i.e., the shift functor actually decreases the grading by 1.
\end{rmk}

This paper is organized as follows. In Section 1, we review some basic material in contact geometry including convex surface theory and bypasses. In Section 2, we recall the classical Pontryagin-Thom construction for closed manifold $M$, and generalize it to the case when $\bdry M$ is nonempty. As an application, we define the Hopf invariant $\pi_3(S^2)$. Finally, we give the proof of Theorem~\ref{MainThm1} and Theorem~\ref{MainThm2} in Section 3.

\section{Contact geometry preliminaries}

\subsection{Convex surfaces}

Let $(M,\xi)$ be a contact manifold. A properly embedded arc $\gamma \subset M$ is {\em Legendrian} if $T_x\gamma \subset \xi_x$ for any $x\in\gamma$. A closed oriented surface $\Sigma \subset M$ is {\em convex} if there exists a contact vector field $v$ transverse to $\Sigma$, i.e., the flow of $v$ preserves $\xi$. In particular, we always assume that $\bdry M$ is convex if nonempty.

Given a convex surface $\Sigma$, we define the {\em dividing set} $\Gamma_\Sigma \coloneqq \{x\in\Sigma~|~v(x)\in\xi_x\}$, where $v$ is a contact vector field transverse to $\Sigma$. The {\em characteristic foliation} $\Sigma_\xi$ is a singular foliation on $\Sigma$ obtained by integrating the singular line field $T\Sigma\cap\xi$. We summarize basic properties of dividing set as follows.

\be
\item{$\Gamma_\Sigma$ is a nonempty smooth 1-dimensional submanifold of $\Sigma$.}

\item{$\Gamma_\Sigma$ is transverse to $\Sigma_\xi$.}

\item{The isotopy class of $\Gamma_\Sigma$ does not depend on the choice of the transverse contact vector field $v$.}
\ee

It is not hard to see that if two contact structures induce the same characteristic foliation on $\Sigma$, then they are isotopic in a neighborhood of $\Sigma$. In fact, E. Giroux~\cite{Gi} showed that one needs much less information --- only the dividing set --- to determine the isotopy class of contact structures in a neighborhood of convex surface. This is the content of the following {\em Giroux's Flexibility Theorem}.

\begin{thm}[Giroux] \label{Flex}
Let $\Sigma$ be a convex surface with characteristic foliation $\Sigma_\xi$, $v$ be a contact vector field transverse to $\Sigma$, and $\Gamma_\Sigma$ be the dividing set. If $\mathscr{F}$ is another singular foliation on $\Sigma$ divided by $\Gamma_\Sigma$, then there exists an isotopy $\phi_t, t\in [0,1]$ such that

\be

\item{$\phi_0=id$ and $\phi_t|_{\Gamma_\Sigma}=id$ for all $t$.}

\item{$v$ is transverse to $\phi_t(\Sigma)$ for all $t$.}

\item{The characteristic foliation on $\phi_1(\Sigma)$ is $\mathscr{F}$.}

\ee
\end{thm}

\subsection{Bypasses}

Following \cite{Ho2}, let $\Sigma$ be a convex surface. A {\em bypass} $D$ on $\Sigma$ is a convex disk with Legendrian boundary $\bdry D=\alpha\cup\beta$ such that the following conditions hold:

\be

\item{$\alpha=\Sigma\cap D$.}

\item{$\Gamma_\Sigma\cap\alpha=\{p_1,p_2,p_3\}$, where $p_1,p_2,p_3$ are distinct points.}

\item{$\alpha\cap\beta=\{p_1,p_3\}$.}

\item{for an appropriate orientation of $D$, $p_1$ and $p_3$ are both positive elliptic singular points of $D$, $p_2$ is a negative elliptic singular point of $D$, and all the singular points along $\beta$ are positive and alternate between elliptic and hyperbolic.}

\ee

\begin{figure}[h]
  \begin{overpic}[scale=.45]{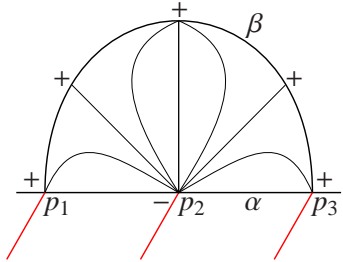}
  \put(11,15){\small{$p_1$}}
  \put(51,15){\small{$p_2$}}
  \put(90.7,15){\small{$p_3$}}
  \put(5,22){\small{$+$}}
  \put(92,22){\small{$+$}}
  \put(43.5,15.2){\small{$-$}}
  \put(83,52){\small{$+$}}
  \put(49,72.3){\small{$+$}}
  \put(14,52){\small{$+$}}
  \put(72,68){\small{$\beta$}}
  \put(71,15){\small{$\alpha$}}
  \end{overpic}
  \caption{A bypass}
\end{figure}

\begin{rmk}
One can easily decrease the Thurston-Bennequin invariant by stabilizing a Legendrian arc. However, the converse is not always possible in a contact manifold. Observe that in the definition of a bypass, we need to increase the Thurston-Bennequin invariant by 1. Hence most bypasses do not come for free. In this paper, we do not worry about the existence of bypasses because we will attach bypasses from outside of the contact manifold.
\end{rmk}

Given a convex surface and a bypass as above, we now describe a bypass attachment.

\begin{lemma}[Honda] \label{bypasslem}
Assume $D$ is a bypass for a convex surface $\Sigma$. Then there exists a neighborhood of $\Sigma\cup D\subset M$ diffeomorphic to $\Sigma\times[0,1]$, such that $\Sigma_i=\Sigma\times\{i\},i=0,1$, are convex, and $\Gamma_{\Sigma_1}$ is obtained from $\Gamma_{\Sigma_0}$ by performing the bypass attachment operation depicted in Figure~\ref{BypassAttach} in a neighborhood of the attaching Legendrian arc $\alpha$.
\end{lemma}

\begin{figure}[h]
  \begin{overpic}[scale=.55]{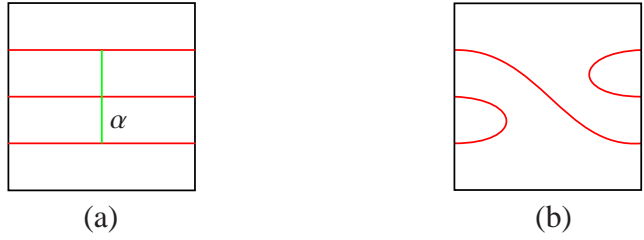}
  \put(16.2,10){\small{$\alpha$}}
  \put(12,-5.5){(a)}
  \put(83,-5.5){(b)}
  \end{overpic}
  \newline
  \caption{Bypass attachment: (a) dividing curves on $\Sigma_0$ and the Legendrian arc of attachment $\alpha$; (b) dividing curves on $\Sigma_1$.}
  \label{BypassAttach}
\end{figure}

In practice, we construct a neighborhood of $\Sigma \cup D$ with a contact structure given by the bypass attachment as follows. Let $D\times[-\epsilon,\epsilon]$ be a thickening of $D$ with an invariant contact structure in the $[-\epsilon,\epsilon]$-direction, where $\epsilon>0$ is small. Then a neighborhood of $\Sigma \cup D$ can be obtained by rounding the corners of $\Sigma \cup (D\times[-\epsilon,\epsilon])$. A more precise construction will be given in Section 3.

\section{The Pontryagin-Thom construction}

\subsection{The Pontryagin-Thom construction for closed manifolds}

The Pontryagin-Thom construction is designed to study homotopy types of smooth maps $f:M\to S^n$, where $M$ is a closed manifold. The idea is that instead of working with maps between manifolds, we study framed submanifolds of $M$ associated with these maps and framed cobordism between them. Throughout this paper, we always assume $M$ is 3-dimensional and $n=2$.

Fix a Riemannian metric on $M$. Let $L \subset M$ be a link. A {\em framing} of $L$ is the homotopy class of a smooth function $\sigma$ which assigns to each point $x\in L$ a basis $\{v_1(x),v_2(x)\}$ of the orthogonal complement of $T_xL$ in $T_xM$. We call the pair $(L,\sigma)$ a {\em framed link}. Two framed links $(L,\sigma)$ and $(L',\sigma')$ are {\em framed cobordant} if there exists a framed
surface $(\Sigma,\delta)$ in the 4-manifold $M\times [0,1]$ such that $(\Sigma,\delta)|_{M\times 0}=(L,\sigma)$ and $(\Sigma,\delta)|_{M\times1}=(L',\sigma')$, where the framing $\delta$ is the homotopy class of a smooth function which assigns to each point $y\in\Sigma$ a basis of the orthogonal complement of $T_y\Sigma$ in $T_y(M\times[0,1])$.

The main result of Pontryagin-Thom construction is the following theorem. See Chapter 7 of~\cite{Mi} for more details.

\begin{thm} \label{PTconstr}
Let $M$ be a closed 3-manifold. Then there exists a one-to-one correspondence

\center\{smooth maps $f:M \to S^2$ up to homotopy\} $\xleftrightarrow{~~1-1~~}$ \{framed links in $M$ up to framed cobordism\}.

\end{thm}

\begin{proof}[Sketch of proof]
To construct a framed link in $M$ from a smooth map $f:M \to S^2$, let $p \in S^2$ be a regular value of $f$. By choosing a basis $\{v_1,v_2\}$ of $T_pS^2$, we obtain a framed link $(L_{f,p},\sigma_{f,p})$ in $M$, where $L_{f,p}=f^{-1}(p)$ and $\sigma_{f,p}(x)$ is the pull-back of $\{v_1,v_2\}$ via the isomorphism $f_\ast:T_xL^\bot \to T_pS^2,\forall x \in L$.

Conversely, let $(L,\sigma)$ be a framed link in $M$. Identify an open tubular neighborhood $N(L)$ of $L$ with $L\times \mathbb{R}^2$ via $\sigma$. Choose a smooth map $\phi:\mathbb{R}^2 \to S^2$ which maps every $x$ with $||x|| \geq 1$ to a base point $y \in S^2$, and maps the open unit disk $||x||<1$ diffeomorphically\footnote{For example, $\phi(x)=\pi^{-1}(x/\lambda(||x||^2))$, where $\pi$ is the stereographic projection from $y$ and $\lambda$ is a smooth monotone function with $\lambda(t)>0$ for $t<1$ and $\lambda(t)=0$ for $t \geq 1$.} onto $S^2\setminus\{y\}$. We define a smooth map $f:M \to S^2$ in two steps. First we define $f|_{N(L)}:N(L) \simeq L\times \mathbb{R}^2 \xrightarrow{\pi_2} \mathbb{R}^2 \xrightarrow{\phi} S^2$, where $\pi_2:L\times \mathbb{R}^2 \to \mathbb{R}^2$ is the projection onto the second factor. Then we extend $f|_{N(L)}$ to $f:M \to S^2$ by the constant map $f|_{M \setminus N(L)} \equiv y \in S^2$.

One can show that the above construction in both directions establishes the desired one-to-one correspondence.
\end{proof}

\begin{defn}
Given a smooth map $f:M \to S^2$, we call the framed link $(L_{f,p},\sigma_{f,p})$ constructed above the {\em Pontryagin submanifold associated with $f$}.
\end{defn}

\begin{rmk}
Although the construction of $(L_{f,p},\sigma_{f,p})$ depends on the choice of $p$, its framed cobordism class does not. Compare with the relative Pontryagin-Thom construction discussed in Section~\ref{relPTsection}.
\end{rmk}

However, Theorem~\ref{PTconstr} is still not satisfactory for our purposes because we will be working with contact manifolds with boundary. Before we generalize the Pontryagin-Thom construction to manifolds with boundary, we look at a simple application of Theorem~\ref{PTconstr} which defines the Hopf invariant.

In \cite{H}, Hopf constructed the well-known {\em Hopf map} $\zeta:S^3\to S^2$ using Clifford parallels and showed that
$\zeta$ is essential, i.e., $\zeta$ is not homotopic to a constant map. Applying Theorem~\ref{PTconstr}, we compute the homotopy group $\pi_3(S^2)$ of $S^2$, also known as the {\em Hopf invariant}. It turns out that $\zeta$ corresponds to a generator of $\pi_3(S^2)$.

\begin{lemma} \label{Hopf}
There exists an isomorphism $h:\pi_3(S^2) \xrightarrow{\sim} \mathbb{Z}$
\end{lemma}

\begin{proof}
Since any continuous map $f:S^3 \to S^2$ can be approximated by a smooth map, we can assume that the elements in $\pi_3(S^2)$ are represented by smooth maps. Now it follows immediately from Theorem~\ref{PTconstr} that $\pi_3(S^2)=\{(L,\sigma)\}/\sim$, where $(L,\sigma)\sim(L',\sigma')$ if and only if they are framed cobordant. The group structure on $\{(L,\sigma)\}/\sim$ is defined by $[(L_1,\sigma_1)]+[(L_2,\sigma_2)]=[(L_1 \sqcup L_2,\sigma_1 \sqcup \sigma_2)]$ and $-[(L,\sigma)]=[(L,-\sigma)]$. If $\tilde L$ is a parallel copy of $L$ given by the framing $\sigma$, then we define $n(L,\sigma)$ to be the self-linking number $lk(L,\tilde L)$. Now we define the group homomorphism $h:\pi_3(S^2) \to \mathbb{Z}$ by sending $[(L,\sigma)]$ to $n(L,\sigma)$. It is easy to verify that $h$ is well-defined and is an isomorphism.
\end{proof}

\subsection{The Pontryagin-Thom construction for manifolds with boundary} \label{relPTsection}

Let $M$ be a compact 3-manifold with boundary. Let $f:M \to S^2$ be a smooth map and $p \in S^2$ be a regular value of $f$. The {\em Pontryagin submanifold} $(f^{-1}(p),\sigma_{f,p})$ associated with the pair $(f,p)$ is a framed 1-dimensional submanifold of $M$, i.e., it is the disjoint union of a framed link and a finite collection of framed arcs with endpoints contained in $\bdry M$. Two framed 1-dimensional submanifolds $(L,\sigma)$ and $(L',\sigma')$ of $M$ are {\em relatively framed cobordant} if there exists a framed surface $(\Sigma,\delta)$ in $M\times[0,1]$ such that (i) $(\Sigma,\delta)|_{M\times\{0\}}=(L,\sigma)$, (ii) $(\Sigma,\delta)|_{M\times\{1\}}=(L',\sigma')$, and (iii) $(\Sigma,\delta)|_{\bdry M\times\{t\}}=(L,\sigma)|_{\bdry M\times\{0\}}=(L',\sigma')|_{\bdry M\times\{1\}}$ for any $t\in[0,1]$. We have the following theorem which can be viewed as the relative analogue of Theorem~\ref{PTconstr}.

\begin{thm} \label{relPTconstr}
Let $M$ be a compact 3-manifold with boundary. If $f,f':M \to S^2$ are smooth maps such that $f|_{\bdry M}=f'|_{\bdry M}$, then $f$ is homotopic to $f'$ relative to the boundary if and only if for any $p\in S^2$ which is a common regular value of $f$ and $f'$, $(f^{-1}(p),\sigma_{f,p})$ is relatively framed cobordant to $(f'^{-1}(p),\sigma_{f',p})$.
\end{thm}

\begin{proof}
Let $H:M\times[0,1] \to S^2$ be a homotopy between $f$ and $f'$ relative to the boundary. Generically we can assume $p\in S^2$ is also a regular value of $H$. Hence the Pontryagin submanifold $(H^{-1}(p),\delta_H)$ defines a relative framed cobordism between $(f^{-1}(p),\sigma_{f,p})$ and $(f'^{-1}(p),\sigma_{f',p})$.

Conversely, let $(\Sigma,\delta) \subset M\times[0,1]$ be a relative framed cobordism between $(f^{-1}(p),\sigma_{f,p})$ and $(f'^{-1}(p),\sigma_{f',p})$. Let $\bdry M\times[-1,0] \subset M$ be a collar neighborhood of $\bdry M$ where $\bdry M\times\{0\}$ is identified with $\bdry M$, and $\tilde M$ be the metric closure of $M\setminus (\bdry M\times[-1,0])$. Abusing notation, we shall write $\Sigma$ for $\Sigma\cap(\tilde M\times[0,1])$. As in the proof of Theorem~\ref{PTconstr}, we identify an open tubular neighborhood $N(\Sigma)$ of $\Sigma$ in $\tilde M\times[0,1]$ with $\Sigma\times \mathbb{R}^2$ via $\delta$, and define a smooth map $H_1:\tilde M\times[0,1] \to S^2$ by (i) $H_1|_{N(\Sigma)}:N(\Sigma)\simeq\Sigma\times \mathbb{R}^2 \xrightarrow{\pi_2} \mathbb{R}^2 \xrightarrow{\phi} S^2$ where $\pi_2:\Sigma\times \mathbb{R}^2 \to \mathbb{R}^2$ is the projection onto the second factor, and (ii) $H_1|_{\tilde M\setminus N(\Sigma)} \equiv y \in S^2$. Observe that $H_1|_{\bdry \tilde M\times\{t\}}: \bdry \tilde M\times\{t\} \to S^2$ is homotopic to $f:\bdry M\times\{t\} \to S^2$ for any $t\in[0,1]$, and let $H_2^t:\bdry M\times[-1,0]\times\{t\} \to S^2$ be the homotopy, i.e., $H_2^t|_{s=-1}=H_1|_{\bdry \tilde M\times\{t\}}$ and $H_2^t|_{s=0}=f$, where $s\in[-1,0]$. Define $H_2: \bdry M\times[-1,0]\times[0,1]$ by $H_2(x,s,t)=H_2^t(x,s)$ for $x\in\bdry M$, $s\in[-1,0]$ and $t\in[0,1]$. We construct a map $H:M\times[0,1] \to S^2$ by gluing $H_1$ and $H_2$ along $\bdry M\times\{-1\}\times[0,1]$ which satisfies $H|_{\bdry M\times\{t\}}=f|_{\bdry M}=f'|_{\bdry M}$ for any $t\in[0,1]$. One can verify that $H|_{M\times\{0\}}$ and $H|_{M\times\{1\}}$ are homotopic to $f$ and $f'$ relative to the boundary, respectively, as in the closed case. Hence the conclusion follows.
\end{proof}

\begin{cor}
Let $f,f':M \to S^2$ be smooth maps such that $f|_{\bdry M}=f'|_{\bdry M}$. If $(f^{-1}(p),\sigma_{f,p})$ is relatively framed cobordant to $(f'^{-1}(p),\sigma_{f',p})$ for some common regular value $p$ of $f$ and $f'$, then the same holds for all common regular values of $f$ and $f'$.
\end{cor}

\begin{proof}
This follows immediately from the proof of Theorem~\ref{relPTconstr}.
\end{proof}

Hence in practice, in order to verify that $f$ is homotopic to $f'$ relative to the boundary, it suffices to check the framed cobordant condition for a preferred common regular value.

\begin{rmk}
One can easily generalize Theorem~\ref{relPTconstr} to arbitrary dimension using the same proof.
\end{rmk}

\subsection{The 3-dimensional obstruction class $o_3(\xi,\xi')$ of 2-plane field distributions}

Let $M$ be a compact oriented 3-manifold, and $\xi$ and $\xi'$ be two oriented 2-plane field distributions on $M$ such that $\xi=\xi'$ on $M\setminus B^3$ for a 3-ball $B^3\subset int(M)$. Fix a trivialization of $TM$. Let $G_\xi:M \to S^2$ and $G_{\xi'}:M \to S^2$ be the Gauss maps associated with $\xi$ and $\xi'$, respectively. Take a common regular value $p\in S^2$ of $G_\xi$ and $G_{\xi'}$, and let $(L,\sigma)$ and $(L',\sigma')$ be the Pontryagin submanifolds associated with $(G_\xi,p)$ and $(G_{\xi'},p)$, respectively, i.e., $L=G^{-1}_\xi(p)$ and $L'=G^{-1}_{\xi'}(p)$. By assumption, $(L,\sigma)=(L',\sigma')$ on $M\setminus B^3$. Hence we may focus on the relative framed cobordism classes of $(L,\sigma)|_{B^3}$ and $(L',\sigma')|_{B^3}$. Since $B^3$ is contractible, $L$ is always relatively cobordant to $L'$ but the framing may not extend to the cobordism. To fix this issue, let $C\subset int(B^3)$ be a trivial loop which does not link with $L'$. Observe that $(L,\sigma)$ is relatively framed cobordant to $(L' \sqcup C,\sigma' \sqcup \delta)$ in $B^3$ for some framing $\delta$ of $C$. If $C'$ is a parallel copy of $C$ given by $\delta$, then we define $n(C,\delta)$ to be the self-linking number $lk(C,C')$ with respect to the orientation of $B^3$ inherited from the orientation of $M$.

\begin{defn} \label{obsrtclass}
Let $\xi$ and $\xi'$ be oriented 2-plane field distributions on $M$ such that $\xi=\xi'$ on $M\setminus B^3$ for a 3-ball $B^3\subset M$. We define the 3-dimensional obstruction class $o_3(\xi,\xi')\in\mathbb{Z}/d(\xi)$ to be $n(C,\delta)$ as constructed above modulo $d(\xi)$, where $d(\xi)$ is the divisibility of the Euler class $e(\xi)\in H^2(M,\mathbb{Z})$.
\end{defn}

\begin{rmk}
One can think of $o_3(\xi,\xi')$ as a relative version of the Hopf invariant described in Lemma~\ref{Hopf}.
\end{rmk}

It is easy to see that the definition of $o_3(\xi,\xi')$ is independent of various choices involved, namely, the trivialization of $TM$, the 3-ball $B^3\subset M$, the trivial loop $C$ and the common regular value $p\in S^2$. The independence of the choice of common regular values is slightly nontrivial, so we prove this in the following lemma.

\begin{lemma}
The obstruction class $o_3(\xi,\xi')\in\mathbb{Z}/d(\xi)$ is independent of the choice of $p\in S^2$.
\end{lemma}

\begin{proof}
Let $\hat M=M \cup_{\bdry M} (-M)$ be a closed oriented 3-manifold, where $-M$ is $M$ with the opposite orientation. Glue $G_\xi$ and $G_{\xi'}$ along $\bdry M$ to obtain a smooth map $\hat{G}: \hat{M}\to S^2$ given by:

\begin{equation*}
\hat G(x) = \left\{
\begin{array}{rl}
G_\xi(x) & \text{if } x \in M,\\
G_{\xi'}(x) & \text{if } x \in -M.
\end{array} \right.
\end{equation*}

If $q\in S^2$ is another common regular value of $G_\xi$ and $G_{\xi'}$, then $p$ and $q$ are both regular values of $\hat{G}$. We write $o^p_3(\xi,\xi')$ (resp. $o^q_3(\xi,\xi')$) for the obstruction class to indicate the potential dependence on the choice of $p$ (resp. $q$). According to Proposition 4.1 in \cite{Go}, we have $o^p_3(\xi,\xi')-o^q_3(\xi,\xi')=0 \in \mathbb{Z}/d(\xi)$. Hence $o_3(\xi,\xi')$ is independent of the choice of $p$ modulo $d(\xi)$.
\end{proof}

Using the same argument as in proof of Proposition 4.1 in \cite{Go}, we also obtain the following result.

\begin{prop} \label{Obs}
If $\xi$ and $\xi'$ are two contact structures on $M$ such that $\xi|_{M\setminus B^3}=\xi'|_{M\setminus B^3}$ for some 3-ball $B^3 \subset int(M)$, then $\xi$ is homotopic to $\xi'$ relative to the boundary if and only if $o_3(\xi,\xi')=0 \in \mathbb{Z}/d(\xi)$.
\end{prop}

\section{Proof of Theorem~\ref{MainThm1} and Theorem~\ref{MainThm2}}

\subsection{Proof of Theorem~\ref{MainThm1}}

Now we are ready to compute the relative Pontryagin submanifold associated with the contact 3-manifold $(V,\xi\ast\sigma_\alpha)$ as constructed in Theorem~\ref{MainThm1}.

\begin{proof}[Proof of Theorem~\ref{MainThm1}]
Recall the manifold $V=[-3/4,3/4] \times [-1,1] \times [0,1] \subset \mathbb{R}^3$ with the contact structure $\xi=\ker \lambda$, where $\lambda=\cos(2\pi x)dy-\sin(2\pi x)dz$. Let $\Sigma_t=[-3/4,3/4] \times [-1,1] \times \{t\}$ be a foliation by convex surfaces with respect to the contact vector field $\bdry/\bdry z$ for $t \in [0,1]$. The dividing set $\Gamma_t$ on $\Sigma_t$, $t\in[0,1]$, is the disjoint union of three parallel intervals $(\{1/2\} \times [-1,1] \times \{t\}) \cup (\{0\} \times [-1,1] \times \{t\}) \cup (\{-1/2\} \times [-1,1] \times \{t\})$ which divide $\Sigma_t$ into positive and negative regions. Let $\alpha=[-1/2,1/2] \times \{0\} \times \{1\} \subset \Sigma_1$ be the Legendrian arc along which an $I$-invariant neighborhood of the bypass $D_\alpha=\{(x,y,z)~|~1 \leq z \leq 1+\sqrt{1/4-x^2},y=0\}$ is attached. We choose the characteristic foliation on $D_\alpha$ so that it is half of an overtwisted disk with one negative elliptic singular point at the center and alternating positive elliptic and hyperbolic singular points on the boundary, and the dividing set $\Gamma_{D_\alpha}$ is a semi-circle centered at $(0,0,1)$ with radius $1/4$. By gluing a $\bdry/\bdry y$-invariant neighborhood $D_\alpha\times[-\epsilon,\epsilon]$ of $D_\alpha$ for small $\epsilon>0$ to $(V,\xi)$, we obtain a contact manifold $(V_\alpha,\xi_\alpha)$ with corners where $V_\alpha= V \cup (D_\alpha \times [-\epsilon,\epsilon])$. Abusing notation, we also denote the contact manifold obtained by rounding corners on $D_\alpha \times [-\epsilon,\epsilon] \subset V_\alpha$ by $(V_\alpha,\xi_\alpha)$. By slightly tilting $D_\alpha\times\{-\epsilon\}$ and $D_\alpha\times\{\epsilon\}$, we can further assume that the $\bdry/\bdry z$-direction is transverse to $\bdry_+ V_\alpha$, the top boundary of $V_\alpha$. Observe that, up to isotopy, $\Gamma_{\bdry_+ V_\alpha}$ is as depicted in Figure~\ref{BypassAttach}(b).

Choose a non-positive smooth function $g: V_\alpha \to \mathbb{R}_{\leq 0}$ supported in a neighborhood of $D_\alpha \times [-\epsilon,\epsilon]$ such that the time-$1$ map $\phi_X^1$ of the flow of $X=g \bdry/\bdry z$ sends $V_\alpha$ diffeomorphically onto $V$. We identify $V_\alpha$ with $V$ via $\phi_X^1$, and we denote the contact structure $(\phi_X^1)_\ast(\xi_\alpha)$ by $\xi\ast\sigma_\alpha$, where $\xi\ast\sigma_\alpha$ is known as the contact structure obtain by attaching a bypass along $\alpha$ to $\xi$.

Next, we study the homotopy type of $(V,\xi\ast\sigma_\alpha)$ using the Pontryagin-Thom construction. Let $p=(1,0,0)\in S^2$ be a regular value of the Gauss map $G_{\xi\ast\sigma_\alpha}$ associated with $\xi\ast\sigma_\alpha$, where $TV$ is trivialized by the standard embedding $V \subset \mathbb{R}^3$. Observe that $p$ is also a regular value of the Gauss map $G_{\xi_\alpha}: V_\alpha \to S^2$ associated with $\xi_\alpha$. In order to keep track of the framing of $G^{-1}_{\xi_\alpha}(p)$, we fix another regular value $p'=(1-\delta,\sqrt{2\delta-\delta^2},0)\in S^2$ near $p$ for small $\delta>0$. It is easy to see that $G^{-1}_{\xi_\alpha}(p)$ and $G^{-1}_{\xi_\alpha}(p')$ are two parallel arcs with endpoints contained in $D_\alpha\times\{-\epsilon,\epsilon\}$\footnote{Remember that $D_\alpha\times\{-\epsilon,\epsilon\}$ is slightly tilted so that it is transverse to the $\bdry/\bdry z$-direction.} as depicted in Figure~\ref{bypass1}(a). Without loss of generality, we can assume that the endpoints of $G^{-1}_{\xi_\alpha}(p)$ and $G^{-1}_{\xi_\alpha}(p')$ are contained in the dividing set $\Gamma_{D_\alpha\times\{-\epsilon,\epsilon\}}$. Note that $G_{\xi_\alpha}(x)$ is contained in the unit circle $S^1=\{z=0\} \subset S^2$ if and only if the same holds for $G_{\xi\ast\sigma_\alpha}(\phi^1_X(x))$. By applying the diffeomorphism $\phi^1_X:V_\alpha \to V$, we obtain the Pontryagin submanifold $G^{-1}_{\xi\ast\sigma_\alpha}(p)$ associated with $\xi\ast\sigma_\alpha$ with framing given by $G^{-1}_{\xi\ast\sigma_\alpha}(p')$ as depicted in Figure~\ref{bypass1}(b). This finishes the proof of Theorem~\ref{MainThm1}.

\begin{figure}[h]
  \begin{overpic}[scale=.65]{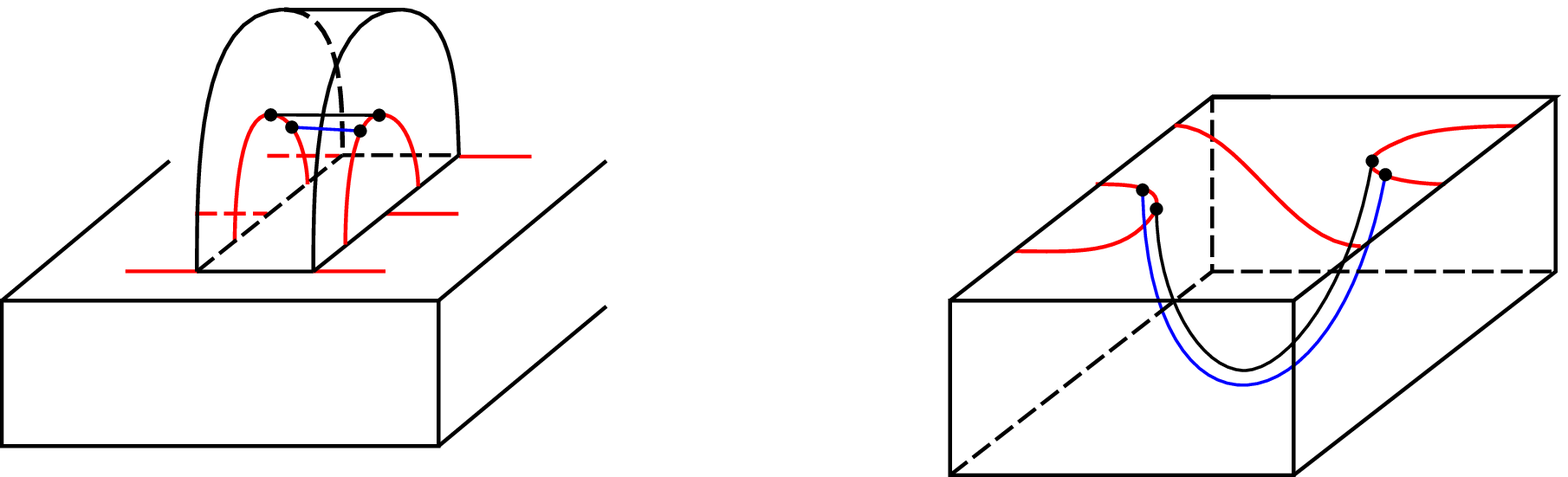}
  \put(24.5,14.5){$-$}
  \put(29.5,18.1){$+$}
  \put(91,20){$+$}
  \put(82,21){$-$}
  \put(78.5,16){$+$}
  \put(68,15){$-$}
  \put(14,-5){(a)}
  \put(76,-5){(b)}
  \end{overpic}
  \newline
  \caption{(a) The Pontryagin submanifold $G^{-1}_{\xi_\alpha}(p)$ contained in $V_\alpha$. (b) The Pontryagin submanifold $G^{-1}_{\xi\ast\sigma_\alpha}(p)$ contained in $V$. The blue arc is a parallel copy of $G^{-1}_{\xi\ast\sigma_\alpha}(p)$ which defines the framing.}
  \label{bypass1}
\end{figure}

\end{proof}

\subsection{Proof of Theorem~\ref{MainThm2}}

We proceed to the proof of Theorem~\ref{MainThm2} which involves three bypass attachments. Our strategy is first to construct a local model for the bypass triangle attachment in $\mathbb{R}^3$, and compute the associated Pontryagin submanifold based on essentially the same methods used in the proof of Theorem~\ref{MainThm1}. Next we identify a neighborhood of the arc of attachment $\alpha$ in $M$ where the bypass triangle is attached with our previously constructed local model, and conclude that the bypass triangle attachment drops $o_3$ by 1.

We first establish a technical lemma which enables us to isotop characteristic foliations on a disk adapted to a fixed dividing set without affecting the Pontryagin submanifold.

\begin{lemma} \label{isotopCharFoli}
Let $(D^2\times[0,1],\xi)$ be a contact 3-manifold with $T(D^2\times[0,1])$ trivialized by the standard embedding $D^2\times[0,1] \subset \mathbb{R}^3$, i.e., $D^2$ is contained in the $xy$-plane and $[0,1]$ is in the direction of the $z$-axis. Suppose the following conditions hold:

\begin{enumerate}

\item {There exists a contact vector field on $D^2\times[0,1]$, with respect to which $D^2\times\{t\}$ are convex and the dividing sets $\Gamma_{D^2\times\{t\}}$ agree for all $t \in [0,1]$.}

\item {The characteristic foliations $\mathscr{F}_{D^2\times\{t\}}$ agree in a neighborhood of $\Gamma_{D^2\times\{t\}}$ for all $t \in [0,1]$.}

\item{The Gauss map $G_\xi$ satisfies: (i) $G_\xi(\Gamma_{D^2\times\{t\}}) \subset \{z=0\} \subset S^2$, (ii) $G_\xi(R_+(D^2\times\{t\})) \subset \{z>0\} \subset S^2$, and (iii) $G_\xi(R_-(D^2\times\{t\})) \subset \{z<0\} \subset S^2$ for any $t \in [0,1]$.}

\item {$p=(1,0,0) \in S^2$ is a regular value of $G_\xi$, and $G_\xi^{-1}(p)$ is disjoint from $\bdry D^2 \times[0,1]$.}

\end{enumerate}

Then $G^{-1}_\xi(p)$ is framed cobordant to $G^{-1}_{\xi_0}(p)$ relative to the boundary, where $\xi_0$ is the $I$-invariant contact structure on $D^2\times[0,1]$ with $\xi_0|_{D^2\times\{0\}} = \xi|_{D^2\times\{0\}}$.
\end{lemma}

\begin{proof}
The conclusion follows from the observation that $G^{-1}_\xi(p) \cap (D^2\times\{t\}) \subset \Gamma_{D^2\times\{t\}}$ for all $t\in[0,1]$, and $\xi$ is $I$-invariant in a neighborhood of $\Gamma_{D^2\times\{0\}} \times [0,1]$ in $D^2\times[0,1]$.
\end{proof}

The following proposition constructs a local model for the bypass triangle attachment explicitly and computes its Pontryagin submanifold.

\begin{prop} \label{localmodel}
Let $T=[-3/4,3/4]\times[-1,1]\times[0,3] \subset \mathbb{R}^3$ be a 3-manifold, $\eta= ker(\cos(2\pi x)dy-\sin(2\pi x)dz)$ be a contact structure on $T$, and $\alpha=\{-1/2 \leq x \leq 1/2, y=z=0\}$ be a Legendrian arc. Then there exists a contact 3-manifold $(T,\eta\ast\triangle_\alpha)$ where $\eta\ast\triangle_\alpha$ is the contact structure obtained from $\eta$ by attaching a bypass triangle along $\alpha$, such that the Pontryagin submanifold $G^{-1}_{\eta\ast\triangle_\alpha}(p)$ is the unknot with framing $-1$ with respect to the standard orientation. Here $p=(1,0,0) \in S^2$ is a regular value of $G^{-1}_{\eta\ast\triangle_\alpha}$.
\end{prop}

\begin{proof}

We construct $(T,\eta\ast\triangle_\alpha)$ and compute its Pontryagin submanifold in three steps corresponding to three bypass attachments $\sigma_\alpha$, $\sigma_{\alpha'}$ and $\sigma_{\alpha''}$ respectively.

\s\n
{\sc Step 1.} We simply use the construction of $(V,\eta\ast\sigma_\alpha)$ \footnote{The contact structure $\eta$ here is the same as $\xi$ in the notation of Theorem~\ref{MainThm1}.} in the proof of Theorem~\ref{MainThm1}. Recall that the Pontryagin submanifold $G^{-1}_{\eta\ast\sigma_\alpha}(p)$ is a framed arc in $V$ as depicted in Figure~\ref{bypass1}(b).

\s\n
{\sc Step 2.} We compute the Pontryagin submanifold associated with the second bypass attachment $\sigma_{\alpha'}$ in two substeps.

\s
{\sc Substep 2.1.} We attach the second bypass in a similar manner. Let $U=[-3/4,3/4] \times [-1,1] \times [1,2] \subset \mathbb{R}^3$ be a contact 3-manifold with contact structure obtained by a $\bdry/\bdry z$-invariant extension of $\eta\ast\sigma_\alpha|_{\Sigma_1}$, where $\Sigma_1=[-3/4,3/4]\times[-1,1]\times\{1\}$. Recall that the second bypass is attached along the Legendrian arc $\alpha'$ as depicted in Figure~\ref{BypassTri}(b). Let $D_{\alpha'}$ be the bypass along $\alpha'$, and $(U_{\alpha'},\eta_{\alpha,\alpha'})$ be the contact 3-manifold obtained by rounding the corners of $U \cup (D_{\alpha'}\times[-\epsilon,\epsilon])$ with the glued contact structure for small $\epsilon>0$. By Lemma~\ref{isotopCharFoli}, we can choose a Legendrian representative $\alpha'$ within its isotopy class such that $ p \notin G_{\eta_{\alpha,\alpha'}}(D_{\alpha'}\times[-\epsilon,\epsilon])$, the image of $D_{\alpha'}\times[-\epsilon,\epsilon]$ under the associated Gauss map $G_{\eta_{\alpha,\alpha'}}$. Since the contact structure remains $I$-invariant away from a neighborhood of $\alpha'$, by pushing $D_{\alpha'}\times[-\epsilon,\epsilon]$ into $U$, we obtain the contact 3-manifold $(U,(\eta \ast \sigma_\alpha|_{\Sigma_1}) \ast \sigma_{\alpha'})$ whose Pontryagin submanifold $G^{-1}_{(\eta \ast \sigma_\alpha|_{\Sigma_1}) \ast \sigma_{\alpha'}}(p)$ is as depicted in Figure~\ref{bypass2}(a).\\

\begin{figure}[h]
  \begin{overpic}[scale=.37]{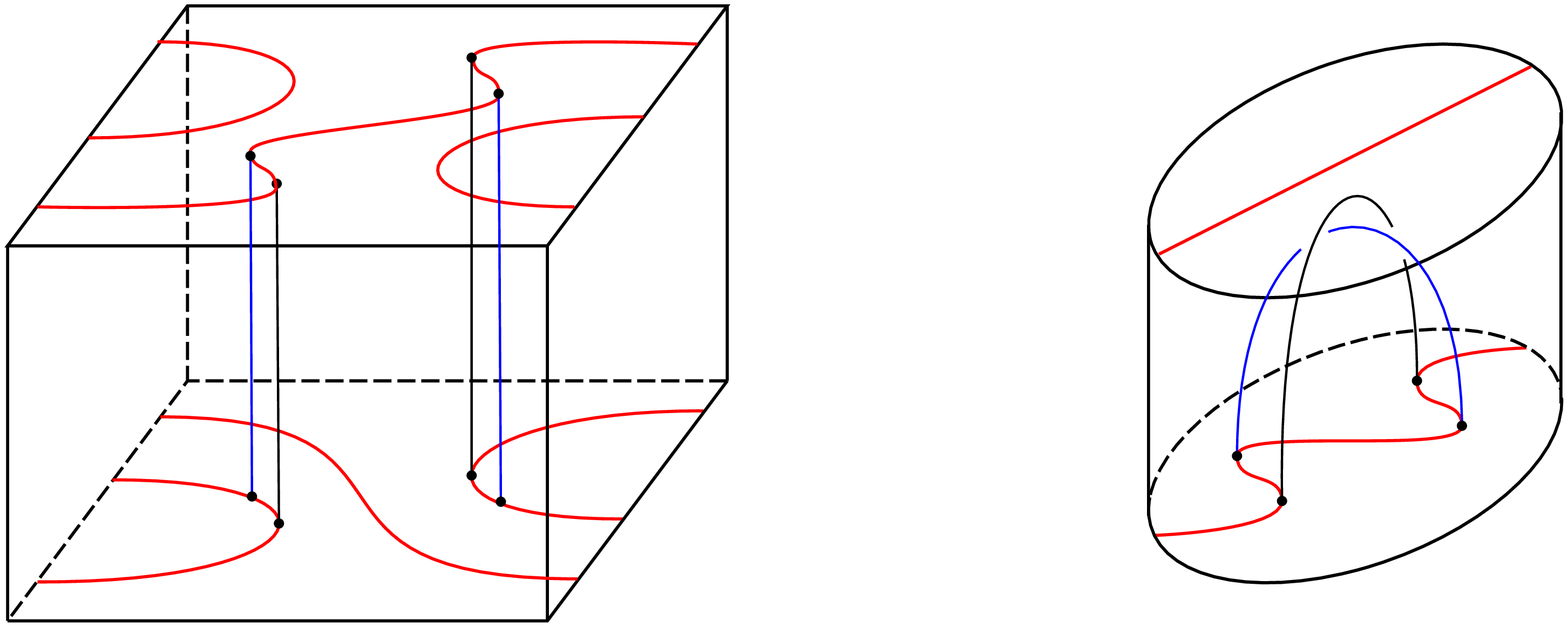}
  \put(86,9){$\gamma$}
  \put(6,4.5){$-$}
  \put(18,4){$+$}
  \put(24.5,10.5){$-$}
  \put(35.5,8.5){$+$}
  \put(16,-5){(a)}
  \put(84,-5){(b)}
  \end{overpic}
  \newline
  \caption{(a) The Pontryagin submanifold $G^{-1}_{(\eta \ast \sigma_\alpha|_{\Sigma_1}) \ast \sigma_{\alpha'}}(p)$ contained in $U$. (b) The Pontryagin submanifold $G^{-1}_\tau(p)$ contained in $N(\gamma)\times[2,3]$.}
  \label{bypass2}
\end{figure}

\s
{\sc Substep 2.2} Let $\gamma \subset \Gamma_{\Sigma_2}$ be the arc containing the endpoints of $G^{-1}_{(\eta \ast \sigma_\alpha|_{\Sigma_1}) \ast \sigma_{\alpha'}}(p)$ on $\Sigma_2=[-3/4,3/4] \times [-1,1] \times \{2\}$, and $N(\gamma)$ be a neighborhood of $\gamma$ on $\Sigma_2$. It is easy to see that there exists an isotopy $\phi_t:N(\gamma) \to N(\gamma)$, $t\in[0,1]$, $\phi_0=id$, such that $p$ is not contained in the image of $N(\gamma)$ under the Gauss map $G_{(\phi_1)_\ast(\eta\ast\sigma_\alpha\ast\sigma_{\alpha'}|_{N(\gamma)})}$. If we define $\Phi:N(\gamma)\times[2,3] \to N(\gamma)\times[2,3]$ by $\Phi(x,t)=(\phi_t(x),t)$ for $x \in N(\gamma)$, $t\in[2,3]$, then we can push-forward a $\bdry/\bdry z$-invariant contact structure $\eta\ast\sigma_\alpha\ast\sigma_{\alpha'}|_{N(\gamma)}$ on $N(\gamma)\times[2,3]$ via $\Phi$ to obtain a new contact structure on $N(\gamma)\times[2,3]$, which we denote by $\tau$. The Pontryagin submanifold $G^{-1}_\tau(p)$ in $N(\gamma)\times[2,3]$ is a framed arc as depicted in Figure~\ref{bypass2}(b). Hence we obtain a contact manifold $(U \cup (N(\gamma)\times[2,3]),((\eta\ast\sigma_\alpha)|_{\Sigma_1}\ast\sigma_{\alpha'}) \cup \tau)$. By rounding the corners of $N(\gamma)\times[2,3]$ and pushing it into $U$ as usual, we obtain the contact 3-manifold which we still denote by $(U,(\eta\ast\sigma_\alpha)|_{\Sigma_1}\ast\sigma_{\alpha'})$ whose associated Pontryagin submanifold $G^{-1}_{(\eta \ast \sigma_\alpha|_{\Sigma_1}) \ast \sigma_{\alpha'}}(p)$ is a framed arc as depicted in Figure~\ref{bypass2end}.

\begin{figure}[h]
    \begin{overpic}[scale=.35]{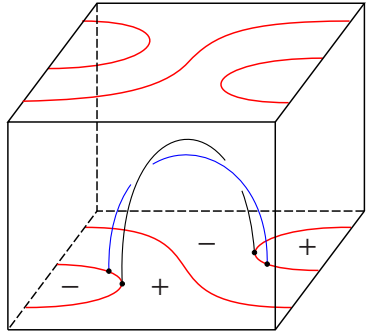}
    \put(15,10){$-$}
    \put(40,10){$+$}
    \put(53,22){$-$}
    \put(81,21){$+$}
    \end{overpic}
    \caption{The Pontryagin submanifold $G^{-1}_{(\eta \ast \sigma_\alpha|_{\Sigma_1}) \ast \sigma_{\alpha'}}(p)$ contained in $U$ after an isotopy.}
    \label{bypass2end}
\end{figure}

\s\n
{\sc Step 3.} We finish the bypass triangle by attaching the third bypass $D_{\alpha''}$ along $\alpha''$ as depicted in Figure~\ref{BypassTri}(c). As in previous steps, let $W=[-3/4,3/4] \times [-1,1] \times [2,3] \subset \mathbb{R}^3$ be a contact 3-manifold with contact structure obtained by a $\bdry/\bdry z$-invariant extension of $\eta\ast\sigma_\alpha\ast\sigma_{\alpha'}|_{\Sigma_2}$.  Again by Lemma~\ref{isotopCharFoli}, we can choose $\alpha''$ so that $p$ is not contained in the image of $D_{\alpha''}\times[-\epsilon,\epsilon]$ under the Gauss map. Hence the same argument as before produces the third contact 3-manifold $(W,(\eta \ast \sigma_\alpha \ast \sigma_{\alpha'}|_{\Sigma_2}) \ast \sigma_{\alpha''})$ whose associated Pontryagin submanifold $G^{-1}_{(\eta \ast \sigma_\alpha \ast \sigma_{\alpha'}|_{\Sigma_2}) \ast \sigma_{\alpha''}}(p)$ is the empty set.

Finally, in order to construct $(T,\eta\ast\triangle_\alpha)$ with the desired properties, we simply let $(T,\eta\ast\triangle_\alpha)=(V,\eta\ast\sigma_\alpha) \cup (U,(\eta\ast\sigma_\alpha|_{\Sigma_1})\ast\sigma_{\alpha'}) \cup (W,(\eta\ast\sigma_\alpha\ast\sigma_{\alpha'}|_{\Sigma_2})\ast\sigma_{\alpha''})$ glued along adjacent faces. It is easy to see that the associated Pontryagin submanifold $G^{-1}_{\eta\ast\triangle_\alpha}(p)$ obtained by gluing the framed arcs from Steps 1, 2, and 3 is the unknot with framing $-1$. See Figure~\ref{PTofTri}.

\begin{figure}[h]
    \begin{overpic}[scale=.45]{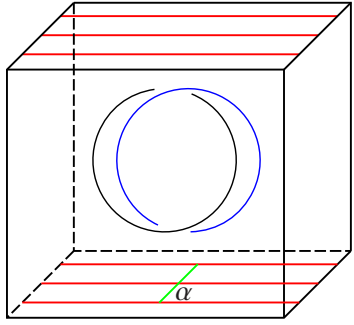}
    \put(49,5.6){\small{$\alpha$}}
    \end{overpic}
    \caption{The Pontryagin submanifold $G^{-1}_{\eta\ast\triangle_\alpha}(p)$. The blue circle is a parallel copy of $G^{-1}_{\eta\ast\triangle_\alpha}(p)$ which defines the framing.}
    \label{PTofTri}
\end{figure}

\end{proof}

\begin{proof}[Proof of Theorem~\ref{MainThm2}]
Let $\alpha\subset \bdry M$ be the Legendrian arc such that $\xi' \simeq \xi\ast\triangle_\alpha$ relative to the boundary, and $N(\alpha)$ be a neighborhood of $\alpha$ on $\bdry M$. Let $\bdry M\times[-1,0] \subset M$ be a collar neighborhood of $\bdry M$ with an $I$-invariant contact structure such that $\bdry M$ is identified with $\bdry M\times\{0\}$. Assume up to a boundary relative isotopy that $\triangle_\alpha$ is supported in $N(\alpha)\times[-2/3,-1/3] \subset int(M)$, i.e., $\xi=\xi'$ on $M \setminus (N(\alpha)\times[-2/3,-1/3])$, and that there exists a contactomorphism $\psi:(N(\alpha)\times[-2/3,-1/3],\xi') \to (T,\eta\ast\triangle_\alpha)$ where $(T,\eta\ast\triangle_\alpha)$ is the local model for a bypass triangle attachment constructed in Proposition~\ref{localmodel}. Without loss of generality, we also choose the trivialization of $TM$ so that its restriction to $N(\alpha)\times[-2/3,-1/3]$ coincides with the pull-back of $T\mathbb{R}^3$ via $\psi$. Let $p=(1,0,0)\in S^2$ be a common regular value of $G_\xi$ and $G_{\xi'}$. Observe that the Pontryagin submanifold $G^{-1}_{\xi'}(p)$ restricted to $N(\alpha)\times[-2/3,-1/3]$ is the unknot with framing $-1$. Since $G^{-1}_\xi(p)$ restricted to $N(\alpha)\times[-2/3,-1/3]$ is the empty set, it follows from Definition~\ref{obsrtclass} that $o_3(\xi,\xi')=-1$ as desired.

In particular, $\xi$ is not homotopic to $\xi'$ relative to the boundary by Proposition~\ref{Obs} since $d(\xi)$ is always even.
\end{proof}

\noindent
{\em Acknowledgements}. The author is grateful to Ko Honda for patient guidance throughout this work. I also thank Francis Bonahon for discussions on the Hopf invariant.

\end{document}